\documentclass[reqno]{amsart}
\usepackage{amsmath,amssymb,amsthm}
\usepackage{fullpage}
\usepackage{enumerate}
\usepackage{subfigure}
\usepackage{tikz}
\usetikzlibrary{intersections}
\usetikzlibrary{3d}
\usetikzlibrary{calc}
\usetikzlibrary{decorations.pathreplacing}

\newtheorem{lemma}{Lemma}
\newtheorem{definition}{Definition}
\newtheorem{theorem}{Theorem}
\newtheorem{conjecture}{Conjecture}
\newtheorem{corollary}{Corollary}

\newcommand{\ZZ}{\mathcal{Z}}
\newcommand{\R}{\mathbb{R}}
\newcommand{\Z}{\mathbb{Z}}

\newcommand{\FF}{\mathcal{F}}
\newcommand{\BB}{\mathcal{B}}
\newcommand{\GG}{\mathbf{G}}
\newcommand{\MM}{\mathbb{M}}
\newcommand{\Mcal}{\mathcal{M}}
\newcommand{\NN}{\mathbb{N}}
\newcommand{\defeq}{:=}

\newcommand{\Ex}[1]{\text{Ex}\left(#1\right)}
\newcommand{\si}[1]{\text{si}\left(#1\right)}
\renewcommand{\SS}{\mathbb{S}}
\newcommand{\e}{\text{e}}
\newcommand{\cl}{\text{cl}}
\newcommand{\ignore}[1]{}

\tikzstyle{dot}=[circle, fill=black]
\tikzstyle{line}=[ultra thick]

\title{Counting matroids in minor-closed classes}

\author{R. A. Pendavingh}
\address{Eindhoven University of Technology, Eindhoven, the Netherlands}
\email{R.A.Pendavingh@tue.nl}

\author{J. G. van der Pol}
\address{Eindhoven University of Technology, Eindhoven, the Netherlands}
\email{jornvanderpol@gmail.com}

\begin{document}

\begin{abstract}
	A flat cover is a collection of flats identifying the non-bases of a matroid. 
	We introduce the notion of cover complexity, the minimal size of such a flat cover, as a measure for the complexity of a matroid, and present bounds on the number of matroids on $n$ elements whose cover complexity is bounded. 
	We apply cover complexity to show that the class of matroids without an $N$-minor is asymptotically small in case $N$ is one of the sparse paving matroids $U_{2,k}$, $U_{3,6}, P_6, Q_6$ or $R_6$,  thus confirming a few special cases of a conjecture due to Mayhew, Newman, Welsh, and Whittle. 
	On the other hand, we show a lower bound on the number of matroids without $M(K_4)$-minor which asymptoticaly matches the best known lower bound on the number of all matroids, due to Knuth.
\end{abstract}

\maketitle

\section{Introduction} In graph theory, there is a well-established theory of asymptotics, often phrased in terms of random graphs (see, e.g.\ \cite{BollobasRandomGraphs}). For example, the Erd\H{o}s-R\'enyi random graph $\GG_{n,\frac{1}{2}}$ on $n$ vertices in which each possible edge is present with probability $\frac{1}{2}$ (independently of any other edge) is connected with probability tending to 1. Similarly, and perhaps slightly more in the flavour of the current paper, the probability that $\GG_{n,\frac{1}{2}}$ contains a fixed subgraph tends to 1.
Such statements are essentially claims about the number of graphs on $n$ vertices that have a certain property compared to the total number of graphs on $n$ vertices, and in particular the asymptotic value of this fraction.

The present paper studies similar such statements in matroid theory. Denoting the set of matroids with ground set $[n]$ by $\mathbb{M}_n$, we consider the limit 
\begin{equation}\label{almost}\lim_{n\rightarrow \infty} \frac{\#\{M\in \mathbb{M}_n : M\text{ has property }\mathcal{P}\}}{\#\mathbb{M}_n}\end{equation}
for a matroid property $\mathcal{P}$.  If the limit exists and is equal to 1, we say that  {\em asymptotically almost every} matroid has property $\mathcal P$.

In their recent paper \cite{MayhewNewmanWelshWhittle2011}, Mayhew, Newman, Welsh, and Whittle observe that `essentially nothing is known about the properties of  ``almost all''  matroids'. They note the lack of a successful model  of random matroids. 
The key word here is `successful', as there is of course a trivial way to create a random matroid on a finite set $E$  uniformly at random:  one draws a subset of the powerset of $E$ repeatedly until one hits upon a set that  satisfies the independence axioms of a matroid on $E$. But this can hardly be called a successful model, since the scarcity of sets ${\mathcal I}\subseteq 2^E$ that satisfy the independence axioms among all subsets of $2^E$ makes it difficult to analyse the properties of such sets ${\mathcal I}$. Where it is straightforward to construct a random graph,  it is not at all obvious what a good model of random matroids should look like. 

In \cite{BansalPendavinghVanderPol2012}, Nikhil Bansal and the present authors proved an upper bound on the number $m_n$ of distinct matroids on $n$ elements, the denominator in \eqref{almost}. Combining our upper bound with the lower bound due to Knuth \cite{Knuth1974}, we have 
\begin{equation}\label{upper_lower}\frac{1}{n}\binom{n}{\lfloor n/2\rfloor}\leq \log m_n \leq \frac{2}{n}\binom{n}{\lfloor n/2\rfloor}(1+o(1)) \qquad\text{as $n\rightarrow \infty$}.\end{equation}
Note that logarithms are taken to the base 2 in  this paper. While the methods for proving these bounds do not immediately yield a satisfactory model of random matroids, we do think that they represent a step in that general direction. Essentially, we prove our upper bound by showing that each matroid on $n$ elements admits a concise description, with a length in bits bounded by the RHS in \eqref{upper_lower}. Sampling from bitstrings of that length until we find a proper compressed description of a matroid is still a somewhat hopeless way to create random matroids, but in view of the lower bound on the number of matroids on $n$ elements, the scarcity of such descriptions is not nearly as bad as before. We hope that by tightening the gap between the upper and lower bound in \eqref{upper_lower}, we will eventually arrive at a satisfactory model of random matroids as well.

Meanwhile, the methods of \cite{BansalPendavinghVanderPol2012} do inspire the following general strategy for showing that asymptotically almost every matroid has a certain property $\mathcal P$. The argument for bounding the size of a  concise description of a matroid $M$ can sometimes be extended to show that if $M$ does not have a property $\mathcal P$, then $M$ has a description of a much shorter length than the general upper bound in \eqref{upper_lower}. Then, the upper bound on the number of matroids without property $\mathcal P$  is much better accordingly. If this upper bound asymptotically becomes much smaller than Knuth's lower bound, we have proved that almost every matroid has propery $\mathcal P$.

In Section \ref{sec:cover}, we describe the nature of our concise matroid descriptions in more detail. We introduce the {\em cover complexity} of a matroid as the size of a smallest such description, and show that the cover complexity has several of the properties that one might wish for in a measure of algorithmic complexity of matroids, such as minor-monotonicity. We describe how a uniform upper bound on the cover complexity of matroids in a certain class yields an upper bound on the number of matroids in that class. We present our key technical tool, the Blow-up Lemma, that bounds the cover complexity of a matroid in terms of the cover complexities of its minors of a certain fixed rank. 

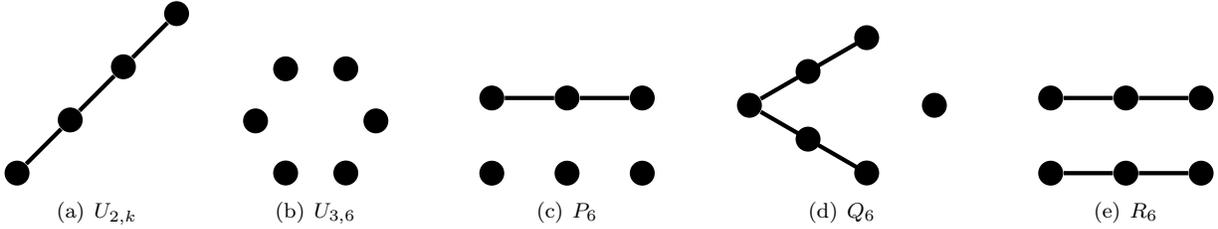
\begin{figure}[tb]
	\subfigure[$U_{2,k}$]{
		\begin{tikzpicture}[scale=0.8]
			\node[dot] (A) {};
			\node[dot, above right of=A] (B) {};
			\node[dot, above right of=B] (C) {};
			\node[dot, above right of=C] (D) {};
			
			\draw[line] (A) -- (B) -- (C) -- (D);
		\end{tikzpicture}
	} \hspace{.015\textwidth}
	\subfigure[$U_{3,6}$]{
		\begin{tikzpicture}[scale=0.8]
			\foreach\d in {0,60, 120, 180, 240, 300}{
				\draw (0,0) + (\d:1) node[dot] {};
			}
		\end{tikzpicture}
	} \hspace{.045\textwidth}
	\subfigure[$P_6$]{
		\begin{tikzpicture}[scale=0.9]
			\node[dot] (A) {};
			\node[dot, right of=A] (B) {};
			\node[dot, right of=B] (C) {};
			\node[dot, below of=A] (D) {};
			\node[dot, below of=B] (E) {};
			\node[dot, below of=C] (F) {};
			
			\draw[line] (A) -- (B) -- (C);
		\end{tikzpicture}
	}\hspace{.045\textwidth}
	\subfigure[$Q_6$]{
		\begin{tikzpicture}[scale=0.9]
			\node[dot] (A) {};
			\draw[line] (A) -- (30:1) node[dot] (B) {} -- (30:2) node[dot](C) {};
			\draw[line] (A) -- (-30:1) node[dot] (D) {} -- (-30:2) node[dot](E) {};
			\draw ($(C)!(A)!(E)$) ++ (0:1) node[dot] (F) {};
		\end{tikzpicture}
	} \hspace{.045\textwidth}
	\subfigure[$R_6$]{
		\begin{tikzpicture}[scale=0.9]
			\node[dot] (A) {};
			\node[dot, right of=A] (B) {};
			\node[dot, right of=B] (C) {};
			\node[dot, below of=A] (D) {};
			\node[dot, below of=B] (E) {};
			\node[dot, below of=C] (F) {};
			
			\draw[line] (A) -- (B) -- (C);
			\draw[line] (D) -- (E) -- (F);
		\end{tikzpicture}
	}
	\caption{\label{fig:P6R6}The matroids of Theorem \ref{thm:asymptotic_small}.}
\end{figure}

Using cover complexity, we were able to show that for several matroids $N$, asymptotically almost every matroid has an $N$-minor. This investigation was prompted by the following conjecture of Mayhew, Newman, Welsh and Whittle \cite{MayhewNewmanWelshWhittle2011}.
\begin{conjecture} \label{conj:paving} Let $N$ be a fixed sparse paving matroid. Asymptotically almost every matroid has an $N$-minor.\end{conjecture}
In Section \ref{sec:small}, we confirm this conjecture for the following sparse paving matroids $N$: the uniform matroids $U_{2,k}$ and  $U_{3,6}$, and the matroids $P_6$, $Q_6$, and $R_3$ (see Figure \ref{fig:P6R6}). 
That is, we prove the following theorem.
\begin{theorem}\label{thm:asymptotic_small} If $N=U_{2,k}$ for some $k \ge 2$, or if $N$ is one of $ U_{3,6}$, $P_6$, $Q_6$ or $R_6$, then
\begin{equation*}\lim_{n \to \infty} \frac{\#\{M\in \MM_n : M\text{ does not have  }N\text{ as a minor}\}}{\#\MM_n} = 0.\end{equation*}
\end{theorem} 
We note that minor-classes of matroids not containing  $U_{2,k}$ for some fixed $k$ occur naturally in the Growth Rate Theorem  of Geelen, Kung, and Whittle \cite{GeelenKungWhittle2009}.
Geelen and Nelson  have investigated the exact number of points and the structure of extremal matroids in such classes \cite{GeelenNelson2010, GeelenNelson2013}.


In Section \ref{sec:large}, we show that there are many matroids without $M(K_4)$ as a minor.
\begin{theorem}\label{theorem:k4}Let $m_n':=\#\{M\in \MM_n\mid M \text{ does not have }M(K_4)\text{ as a minor}\}$. Then 
\begin{equation*}
	\log m'_n\geq \frac{1}{n}\binom{n}{\lfloor n/2\rfloor}(1-o(1)) \qquad \text{as $n\rightarrow\infty$}.
\end{equation*}
\end{theorem}
This does not prove or disprove Conjecture \ref{conj:paving} for the case that $N=M(K_4)$; it does prove that if the conjecture holds in this case, it does so in a less overwhelming manner than for the matroids in Theorem \ref{thm:asymptotic_small}.
Asymptotically, our lower bound on the number of matroids without $M(K_4)$-minor matches Knuth's lower bound on the number of all matroids (Theorem \ref{theorem:knuth}). 
This is not a coincidence, since we show our bound by arguing that many of the sparse paving matroids whose number constitutes Knuth's lower bound, do not have $M(K_4)$ as a minor.

Finally, we consider further directions and pose a challenge in Section \ref{sec:outtro}.

\section{Notation and preliminaries}
\subsection{General notation} We write $[n]:=\{1,\ldots, n\}$, and put
\begin{equation*}
\binom{E}{r}:=\{X\subseteq E : |X|=r\}.
\end{equation*}

Throughout the paper, we take logarithms to the base 2.

\subsection{Matroids}
In general, we use matroid terminology and notation as in Oxley \cite{OxleyBook}. If $n$ is a natural number, then  we put
\begin{equation*}
\MM_n:=\{M\text{ a matroid} : E(M)=[n] \}\quad\text{and}\quad\MM_{n,r}:=\{M\text{ a matroid} : E(M)=[n], r(M)=r \}.
\end{equation*}
Also, we write $m_n:=|\MM_n|$ and $m_{n,r}:=|\MM_{n,r}|$ for the cardinalities for these classes. If $N$ is a matroid, then the set of matroids without an $N$-minor is written $\Ex{N}$. 

By a  {\em non-basis} of $M$, we mean a set $X\subseteq E(M)$ of cardinality $r(M)$ that is not a basis.
A matroid $M$ is {\em paving} if each circuit $C$ of $M$ has $|C|\geq r(M)$, and $M$ is {\em sparse paving} if both $M$ and its dual $M^*$ are paving. 
Equivalently, a matroid $M$ is sparse paving if and only if each non-basis of $M$ is a circuit-hyperplane.
We will use the following lower bound on the number of matroids due to Knuth \cite{Knuth1974}, which is based on counting the number of sparse paving matroids in $\MM_{n}$. 
\begin{theorem} \label{theorem:knuth}$\log m_n\geq \frac{1}{n}\binom{n}{\lfloor n/2\rfloor}$.\end{theorem}
A proof for this theorem is presented in section \ref{sec:large}.
\subsection{Bounds on binomial coefficients}

We will frequently use the following standard bounds on binomial coefficients:
\begin{equation}\label{eq:binom_bound1}
	\binom{n}{r} \le \left(\frac{\e n}{r}\right)^r,
\end{equation}
and
\begin{equation}\label{eq:binom_bound2}
	\frac{2^n}{\sqrt{n}} \sqrt{\frac{2}{\pi}}(1-o(1))	 \le \binom{n}{\lfloor n/2 \rfloor} \le \frac{2^n}{\sqrt{n}} \sqrt{\frac{2}{\pi}}.
\end{equation}

\section{\label{sec:cover}Cover complexity of a matroid}
\subsection{Flat covers and cover complexity}

If $X$ is a dependent set in a matroid $M$, and $F$ is a flat of $M$, then we say that $F$ {\em covers} $X$ if $|X \cap F| > r_M(F)$. Thus, $F$ acts as a witness for the dependence of $X$. We say that a collection of flats {\em covers} the matroid $M$ if each non-basis in $M$ is covered by some flat in the collection. We propose the number of flats needed to cover a matroid as a measure for the complexity of this matroid. This inspires the following definitions.
\begin{definition}[Flat cover]
	If $M$ is a matroid, then $\ZZ \subseteq \FF(M)$ is called a {\em flat cover} of $M$ if each non-basis is covered by some $F \in \ZZ$.
\end{definition}
\begin{definition}[Cover complexity]
	If $M$ is a matroid, then its {\em cover complexity} $\kappa(M)$ is defined as the minimum size of a flat cover of $M$,
	\begin{equation*}
		\kappa(M) \defeq \min\{|\ZZ| : \ZZ \subseteq \FF(M) \text{ is a flat cover of $M$}\}.
	\end{equation*}
\end{definition}
The following three lemmas show that the formal notion of cover complexity captures some of the intuition on the 'complexity' of individual matroids. The dual of a matroid $M$ is 'as complex' as $M$, and has the same cover complexity as $M$. Minors of $M$ are 'simpler' than $M$, and so have lower cover complexity. Knowing $M\setminus e$ and $M/e$ suffices to reconstruct $M$, and so the sum of their complexities bounds the complexity of $M$.
\begin{lemma}\label{lemma:cc1}
	If $M$ is a matroid, then $\kappa(M) = \kappa(M^*)$.
\end{lemma}
\begin{proof}
As $M^{**}=M$, it suffices to show that $\kappa(M^*) \le \kappa(M)$. Suppose that $M \in \MM_{n,r}$. Let $\ZZ$ be a minimum-size flat cover of $M$ and let  $\ZZ^*:= \{\cl^*(E\setminus F) : F \in \ZZ\}$.
We show that $\ZZ^*$ is a flat cover of $M^*$. Consider a non-basis $X$ of $M^*$. Then $E\setminus X$ is a non-basis of $M$, so that there exists an $F \in \ZZ$ such that $|(E\setminus X) \cap F| > r(F)$. It follows that $\cl^*(E\setminus F)\in \ZZ^*$ and 
\begin{equation*}
	|X \cap \cl^*(E\setminus F)| \ge |X \cap (E\setminus F)| = |E \setminus F| - |E \setminus X| + |F \setminus X| > |E \setminus F| - r + r(F) = r^*(F). \qedhere
\end{equation*}
\end{proof}
\begin{lemma}\label{lemma:cc3}
	If $N \preccurlyeq M$, then $\kappa(N) \le \kappa(M)$.
\end{lemma}
\begin{proof}
	We will show that the theorem holds when $N = M\setminus\{e\}$ for some element $e$. It then follows from Lemma \ref{lemma:cc1} that the same is true for single element contractions, and the full claim follows from a standard induction argument. Let $\ZZ$ be a flat cover for $M$. We claim that $\ZZ' \defeq \{F\setminus\{e\} : F \in \ZZ\}$ is a flat cover for $N$. We consider two cases.
	
	If $N$ has rank $r$, then the non-bases of $N$ are dependent $r$-sets avoiding $e$. As this is a subset of the set of non-bases of $M$, it follows that $\ZZ'$ is a flat cover for $N$.
	
	If, on the other hand, $N$ has rank $r-1$, then non-bases in $N$ are those $(r-1)$-subsets $X$ such that $X\cup\{e\}$ is a non-basis in $M$. If $F$ covers $X\cup\{e\}$, then
			\begin{equation*}
				|X \cap (F\setminus \{e\})| =
				\begin{cases}
					|X \cup \{e\} \cap F| - 1 \ge r_M(F) > r_M(F\setminus\{e\} = r_N(F\setminus\{e\}) & \text{if $e \in F$}, \\
					|X \cup \{e\} \cap F| > r_M(F) = r_M(F\setminus\{e\} = r_N(F\setminus\{e\}) & \text{if $e \not\in F$}.  \qedhere
				\end{cases}
			\end{equation*}
\end{proof}

\begin{lemma}\label{lemma:cc2}
	Let $M$ be a matroid and let $e\in E(M)$. If $e$ is neither a loop nor a coloop of $M$, then $\kappa(M) \le \kappa(M/e) + \kappa(M\setminus e)$.
\end{lemma}
\begin{proof}
	Let $\ZZ_{M\setminus e}$ be minimal flat cover for $M\setminus e$ and let $\ZZ_{(M/e)^*}$ be a minimal flat cover for $(M/e)^*$. Assume that $e$ is neither a loop nor a coloop of $M$.
	We claim that
	\begin{equation*}
		\ZZ \defeq \{\cl_M(F): F\in  \ZZ_{M\setminus e}\} \cup \{ \cl_M(E\setminus(\cl_M^*(F))) : F \in \ZZ_{(M/e)^*}\}
	\end{equation*}
	is a flat cover for $M$. To see this, consider a nonbasis $X$ of $M$. If $e\not\in X$, then $X$ is a nonbasis of $M\setminus e$ and hence there is a flat $F\in \ZZ_{M\setminus e}$ that covers $X$ in $M\setminus e$.
	Then $\cl_M(F)\in \ZZ$ is a flat of $M$ covering $X$. If $e\in X$, then $E\setminus X$ is a nonbasis of $(M/e)^*$, and hence there is a flat $F\in  \ZZ_{(M/e)^*}$ covering $E\setminus X$ in $(M/e)^*$. Then 
	$\cl^*_M(F)$ covers $E\setminus X$ in $M^*$, and $\cl_M(E\setminus(\cl_M^*(F))\in \ZZ$ is a flat of $M$ covering $X$. Hence $\kappa(M)\leq |\ZZ|\leq |\ZZ_{M\setminus e}|+|\ZZ_{(M/e)^*}|=\kappa(M/e) + \kappa(M\setminus e)$.	 
\end{proof}

\subsection{Bounding the size of a class by cover complexity}
If $M$ is a matroid and $\ZZ$ is a flat cover of $M$, then $M$ is fully determined by its ground set $E(M)$, the rank $r(M)$ and the collection $\{(F, r_M(F)) : F \in \ZZ\}$. As $\kappa(M)$ is the least number of flats in such a flat cover $\ZZ$, it is intuitively clear that we can put a bound on the number of matroids in some class $\Mcal$, if we are able to uniformly bound $\kappa(M)$ for $M \in \Mcal$ from above. 
\begin{lemma}\label{lemma:upperbound_by_kappa}
	Suppose that there exists a function $f:\NN \to \R_+$ such that $\kappa(M) \le f(n)2^n$ for all $M \in \Mcal \cap \MM_n$ and $n$ sufficiently large. If $f(n) \le 1$ eventually and $1/f(n) = O(n^\alpha)$ for some positive constant $\alpha$, then
	$ \log |\Mcal \cap \MM_n| = O\left(f(n)2^n \log(n)\right)$.
\end{lemma}
\begin{proof}
	Let $k_n \defeq f(n) 2^n$. As each matroid in $\Mcal\cap\MM_n$ is uniquely determined by the collection $\{(F,r_M(F)) : F \in \ZZ\}$, which is a subset of $2^{[n]} \times \{0,1, \ldots, n\}$ of size at most $k_n$, it follows that
	\begin{equation*}
		|\Mcal\cap\MM_n| \le \sum_{i=0}^{k_n} \binom{2^n (n+1)}{i} \le k_n \binom{2^n(n+1)}{k_n} \le k_n \left(\frac{\e 2^n (n+1)}{k_n}\right)^{k_n}.
	\end{equation*}
	Upon taking logarithms and using the fact that $k_n \to \infty$ as $n \to \infty$, we find
	\begin{equation*}
		\log |\Mcal \cap \MM_n| \le \log(k_n) + k_n \log \left(\frac{\e 2^n (n+1)}{k_n}\right) =  k_n \log \left(\frac{\e 2^n (n+1)}{k_n}\right)(1+o(1)),
	\end{equation*}
	and the conclusion of the lemma follows by substituting the value of $k_n$.
\end{proof}
\subsection{Fractional cover complexity}
Let $M = (E,\BB)$ be a matroid with set of flats $\FF$, and consider the following linear programme,
\begin{equation}\label{eq:cover_complexity_lp}
	\begin{array}{rll}
		\text{min}	& \sum\limits_{F \in \FF} z_F 					& \\
		\text{s.t.}	& \sum\limits_{F : \text{$F$ covers $X$}} z_F \ge 1		& \forall X \in \binom{E}{r} \setminus \BB, \\
				& z_F \ge 0								& \forall F \in \FF.
	\end{array}
\end{equation}

\begin{definition}[Fractional cover complexity]
	$z \in \R_+^\FF$ is called a {\em fractional cover} if it is a feasible solution to the LP \eqref{eq:cover_complexity_lp}. The value $\kappa^*(M)$ of the LP is called the {\em fractional cover complexity} of $M$.
\end{definition}

If $\ZZ$ is a flat cover and $z_F$ is the indicator variable for $F \in \ZZ$, then $z \in \R_+^\FF$ is a feasible solution to the LP \eqref{eq:cover_complexity_lp} and hence a fractional cover. It follows that $\kappa^*(M) \le \kappa(M)$. Using a standard randomised rounding technique (see e.g.\ \cite{AlonSpencerBook}), we may also put an upper bound on $\kappa(M)$ in terms of $\kappa^*(M)$.

\begin{lemma}\label{lemma:kappa_rounding}
	For $M \in \MM_{n,r}$, we have $\kappa(M) \le \kappa^*(M)\left(\ln(\binom{n}{r}/\kappa^*(M)) + 1\right)$.
\end{lemma}

\begin{proof}
	We use a probabilistic method to prove the existence of a cover of at most the specified size. Let $z \in \R^{\FF}$ be a solution to the LP \eqref{eq:cover_complexity_lp}, and let $\kappa^*(M)$ be its value. Let $p \in \R^{\FF}$ be the probability distribution on $\FF$ with $p_F = z_F / \kappa^*(M)$, and let $\ZZ_0 \subseteq \FF$ be the result of drawing $m \defeq \lceil \kappa^*(M) \ln\left(\binom{n}{r}/\kappa^*(M)\right)\rceil$ objects (independently, with repetitions) from $\FF$ according to $p$.
	
	The probability that any fixed non-basis is not covered by $\ZZ_0$ is
	\begin{equation}\label{eq:prob_not_covered}
		\left(1-\sum_{F:\text{$F$ covers $X$}}\frac{z_F}{\kappa^*(M)}\right)^m \le \left(\e^{-1/\kappa^*(M)}\right)^m = \frac{\kappa^*(M)}{\binom{n}{r}},
	\end{equation}
	where we used the inequality $1-x \le \e^{-x}$ for all $x$. Let $\mathcal{N}$ be the collection of non-bases not covered by $\ZZ_0$. By Equation \eqref{eq:prob_not_covered}, the number of elements in $\mathcal{N}$ is in expectation at most $\kappa^*(M)$. The collection $\ZZ \defeq \ZZ_0 \cup\{\cl(X) : X \in \mathcal{N}\}$ is a flat cover, and the number of elements in $\ZZ$ is in expectation at most $m + \kappa^*(M)$. This proves the existence of a flat cover of size at most $m + \kappa^*(M)$.
\end{proof}

Combining Lemma \ref{lemma:upperbound_by_kappa} and Lemma \ref{lemma:kappa_rounding}, we obtain the following bound on the size of $\Mcal$ in terms of $\kappa^*(M)$.

\begin{lemma}\label{lemma:upperbound_by_kappa*}
	Suppose that there exists a function $f:\NN \to \R_+$ such that $\kappa^*(M) \le f(n)2^n$ for all $M \in \Mcal \cap \MM_n$ and $n$ sufficiently large. If $f(n)\log(n) \le 1$ eventually, and $1/f(n) = O(n^\alpha)$ for some positive constant $\alpha$, then
	$ \log |\Mcal \cap \MM_n| = O\left(f(n)2^n \log^2(n)\right)$.
\end{lemma}

\begin{proof}
	In view of Lemma \ref{lemma:kappa_rounding} and Equation \eqref{eq:binom_bound2}, we have
	\begin{equation*}
		\kappa(M) \le \kappa^*(M)\ln\left(\binom{n}{\lfloor n/2 \rfloor}/\kappa^*(M) + 1\right) \le \kappa^*(M)\ln\left(\frac{2^n}{\kappa^*(M)} + 1\right)
	\end{equation*}
	for $M \in \Mcal \cap \MM_n$. As the function $q \mapsto q \ln(1/q)$ is increasing on $[0,\e^{-1}]$, we find
	\begin{equation*}
		\kappa(M) \le f(n) 2^n \ln\left(\frac{1}{f(n)} + 1\right) = O(f(n)\log(n)2^n).
	\end{equation*}
	The claim now follows from an application of Lemma \ref{lemma:upperbound_by_kappa}.
\end{proof}

Let $\Mcal$ be a class of matroids that is closed under taking contractions. We refer to the following lemma as the Blow-Up Lemma, as we will use it to convert upper bounds on $\kappa^*(M)$ for $M \in \Mcal$ of fixed (small) rank to matroids in $\Mcal$ of general rank. Its proof is based on the fact that the existence of small fractional covers for matroids of low rank can be used to construct small fractional covers of matroids of higher rank.

\begin{lemma}[Blow-Up Lemma]\label{lemma:blow_up}
	Let $\Mcal$ be a class of matroids that is closed under taking contractions. Then for any $t < r < n$, we have
	\begin{equation}
		\frac{1}{\binom{n}{r}}\max\{\kappa^*(M) : M \in \Mcal \cap \MM_{n,r}\} \le \frac{1}{\binom{n-t}{r-t}}\max\{\kappa^*(M) : M \in \Mcal \cap \MM_{n-t,r-t}\}.
	\end{equation}
\end{lemma}

\begin{proof}
	Let $M \in \Mcal\cap\MM_{n,r}$. We construct a fractional cover $z:\FF(M) \to \R_+$ of bounded cost from a collection of `local' fractional covers $z^S:\FF(M) \to \R_+$, one for each $t$-subset $S$ of $[n]$. Let $S$ be such a set. The construction of $z^S$ depends on whether $S$ is dependent or not in $M$:
	\begin{enumerate}
		\begin{item}
			If $S$ is dependent, put $z^S(F) = 1$ if $F = \cl_M(S)$ and $z^S(F) = 0$ otherwise. Note that $\cl_M(S)$ covers each $r$-set containing $S$.
		\end{item}
		
		\begin{item}
			If $S$ is independent, let $M/S$ be the matroid obtained from $M$ by contracting $S$. 
			Let $z':\FF(M/S)\to\R_+$ be a minimal fractional cover of $M/S$ and let $z^S(F) = z'(F\setminus S)$ if $S \subseteq F$ and $z^S(F) = 0$ otherwise. 
		\end{item}
	\end{enumerate}
	Note that in the latter case, $M/S$ is isomorphic to a matroid in  $\Mcal\cap\MM_{n-t,r-t}$, so that for each $S$ we have 
	$$\sum_{F \in \FF(M)} z^S(F)\leq \max\{\kappa^*(M) : M \in \Mcal \cap \MM_{n-t,r-t}\}.$$
	We now define the fractional cover $z:\FF(M)\rightarrow \R_+$, putting
	\begin{equation*}
		z(F) \defeq \frac{1}{\binom{r}{t}} \sum_{S \in \binom{[n]}{t}} z^S(F).
	\end{equation*}
	This is indeed a fractional cover, as for each non-basis $X$ we have
	\begin{equation*}
		\sum_{F : \text{$F$ covers $X$}} z(F) = \frac{1}{\binom{r}{t}} \sum_{F : \text{$F$ covers $X$}} \sum_{S \in \binom{[n]}{t}} z^S(F) \ge \frac{1}{\binom{r}{t}} \sum_{S \in \binom{X}{t}} \sum_{F : \text{$F$ covers $X$}} z^S(F) \ge 1,
	\end{equation*}
	noting that  $\sum\limits_{F : \text{$F$ covers $X$}} z^S(F) \ge 1$ for each $S \in \binom{X}{t}$. Moreover
		\begin{equation*}
		\kappa^*(M) \le \sum_{F \in \FF(M)} z(F) = \frac{1}{\binom{r}{t}} \sum_{S \in \binom{[n]}{t}} \sum_{F \in \FF(M)} z^S(F) \le \frac{\binom{n}{t}}{\binom{r}{t}}\max\{\kappa^*(M') : M' \in \Mcal\cap\MM_{n-t,r-t}\}.
	\end{equation*}
	The claim then follows from the identity $\frac{\binom{n}{t}}{\binom{r}{t}} = \frac{n!(r-t)!}{(n-t)!r!} = \frac{\binom{n}{r}}{\binom{n-t}{r-t}}$.
\end{proof}
We will apply the Blow-Up Lemma for counting matroids in contraction-closed classes, as follows.
\begin{theorem} \label{theorem:count}Let $\Mcal$ be a contraction-closed class of matroids. If for some natural number $s$, and positive constants $c$ and $\epsilon$ we have 
$\max\{\kappa(M)\mid M\in \Mcal \cap \MM_{n,s}\}\leq  cn^{s-1-\epsilon}$, then
$$\log |\Mcal \cap \MM_{n}|\leq  O(\frac{1}{n^{3/2+\epsilon}}2^n\log^2(n))  \text{ as } n\rightarrow\infty.$$
\end{theorem}
\proof Using the Blow-Up Lemma with $t = r-s$, we find that
\begin{equation*}
	\kappa^*(M) \le \frac{cn^{s-1-\epsilon}}{\binom{n-r+s}{s}}\binom{n}{r}
\end{equation*}
for each $M \in \Mcal \cap \MM_{n,r}$. Maximising over $r$ yields
\begin{equation}\label{eq:small_excluded_allr}
	\kappa^*(M) \le O\left(\frac{1}{n^{1+\epsilon}}\binom{n}{\lfloor n/2 \rfloor}\right) =O\left(\frac{1}{n^{3/2+\epsilon}}2^n\right),
\end{equation}
for all $M \in \Mcal \cap \MM_n$. An application of Lemma \ref{lemma:upperbound_by_kappa*} with $f(n) = n^{-3/2-\epsilon}$ then proves the Theorem.
\endproof
As a corollary, we obtain the weaker of the two upper bounds on the number of matroids that appeared in \cite{BansalPendavinghVanderPol2012}. 
\begin{corollary}$\log \log m_n \leq n- (3/2)\log n + 2 \log \log n + O(1)$.
\end{corollary}
\proof We have $\max\{\kappa(M)\mid M\in \Mcal \cap \MM_{n,1}\}\leq 1$, since $\ZZ=\{\cl_M(\emptyset)\}$ is a flat cover for any matroid $M$ of rank 1. Applying the theorem 
to the class $\Mcal$ of all matroids with $s=1$ and $\epsilon=0$, we obtain 
\begin{equation*}\log m_n=\log |\MM_{n}|\leq  O\left(\frac{1}{n^{3/2}}2^n\log^2(n)\right).\end{equation*}
Taking logarithms, the corollary follows.
\endproof
We will apply the Blow-Up Lemma to other minor-closed classes in Section \ref{sec:small}, where we will use that $\Mcal$ is also closed under deletion to bound the fractional cover complexity for matroids of a certain fixed rank.

\subsection{Lower bounding the cover complexity} We show how to give lower bounds on the cover complexity, and we exhibit two classes of matroids with a cover complexity that is exponential in the number of points of the matroid.
\begin{lemma}\label{lemma:cc4}
	If $N$ can be obtained from $M$ by relaxing a circuit-hyperplane, then $\kappa(M) = \kappa(N) + 1$.
\end{lemma}
\begin{proof}
	Suppose that $N$ is obtained from $M$ by relaxing the circuit-hyperplane $H$. As $H$ is the only flat covering $H$, a collection $\ZZ$ is a flat cover of $N$ if and only if $\ZZ\cup \{H\}$ is a flat cover of $M$.  The claim follows.
\end{proof}
It follows immediately that the cover complexity of any matroid is bounded from below by its number of circuit-hyperplanes, with equality if the matroid is sparse paving. This allows to give a lower bound on the cover complexity of spikes.
\subsubsection*{Spikes} Let $S = (G, \mathcal{D})$ be a set system on a ground set of size $n = |G|$, such that no two elements in $\mathcal{D}$ differ in exactly one element. 
There is a unique matroid $\Lambda(S)$ on ground set $\{a_i : i \in G\} \cup \{b_j : j \in G\}$ such that
\begin{enumerate}
	\item for each distinct $i, j \in G$, the set $\{a_i, b_i, a_j, b_j\}$ is both a circuit and a cocircuit, and
	\item for each $X \subseteq G$, the set $\{a_i : i \in X\} \cup \{b_j : j \in G\setminus X\}$ is dependent if and only if $X \in \mathcal{D}$.
\end{enumerate}
A matroid that is constructed in this way is called an {\em $n$-spike}. It is a matroid on $2n$ elements of rank $n$. The sets $\{a_i, b_i\}$ are called the {\em legs} of the spike, and for each $D \in \mathcal{D}$ the set $T(D) \defeq \{a_i : i \in D\} \cup \{b_j : j \in G\setminus D\}$ is called a {\em dependent transversal} of the legs. Our definition is taken from \cite{Geelen2008}.

There are many $n$-spikes, so in view of Lemma \ref{lemma:upperbound_by_kappa}, there must be $n$-spikes with large cover complexity. The following lemma exhibits such a spike.
\begin{lemma}
	Let $S = (G,\mathcal{D})$ be a set system with $G = [n]$ and $\mathcal{D} = \binom{[n]}{\lfloor n/2 \rfloor}$. Then $\kappa(\Lambda(S)) \ge \binom{n}{\lfloor n/2 \rfloor}$.
\end{lemma}

\begin{proof}
	As each dependent transversal is a circuit-hyperplane in $\Lambda(S)$, the number of circuit-hyperplanes in $\Lambda(S)$ is at least $|\mathcal{D}| = \binom{n}{\lfloor n/2 \rfloor}$. The lemma now follows from Lemma \ref{lemma:cc4}.
\end{proof}

Consider the dual programme of the LP \eqref{eq:cover_complexity_lp}:
\begin{equation}\label{eq:cover_complexity_dual}
	\begin{array}{rll}
		\text{max} & \sum\limits_{X \in \binom{E}{r} \setminus \BB} y_X & \\
		\text{s.t.} & \sum\limits_{X: |X \cap F| > r(F)} y_X \le 1 & \forall F \in \FF \\
		 & y_X \ge 0 & \forall X \in \binom{E}{r} \setminus \BB.
	\end{array}
\end{equation}
Write $\mu^*(M)$ for its optimum value. If we further restrict the dual variables to take values in $\Z$, then the new optimum value, call it $\mu(M)$, bounds $\mu^*(M)$ from below. By LP duality,
\begin{equation*}
	\mu(M) \le \mu^*(M) = \kappa^*(M) \le \kappa(M).
\end{equation*}
The binary version of the dual programme \eqref{eq:cover_complexity_dual} asks to find a subset of non-bases of maximum size, such that each flat covers at most one chosen non-basis. Thus, any such collection yields a lower bound for the cover complexity of $M$.
\begin{lemma}\label{lemma:dual}
	If $\mathcal{X}$ is a subset of non-bases of $M$ of maximum size, such that each flat covers at most one $X \in \mathcal{X}$, then $\kappa(M) \ge |\mathcal{X}|$.
\end{lemma}




\subsubsection*{Cycle matroids of complete graphs}

We consider $M(K_{r+1})$, the cycle matroid of the complete graph $K_{r+1}$. Write $V := V(K_{r+1})$ and for any $A \subseteq V$ let $\delta(A)$ for the set of edges having one end point in $A$.
\begin{lemma} Let $r>3$. Then the set $$\ZZ:=\{F\in \FF(M(K_{r+1})) \mid r(F)=r-1\}$$ is a flat cover of $M(K_{r+1})$. Moreover, $\kappa(M(K_{r+1})) = |\ZZ| = 2^r - 1$.
\end{lemma}

\begin{proof}
	We first show that $\ZZ$, the set of hyperplanes of $M(K_{r+1})$, is indeed a flat cover of $M(K_{r+1})$. 
	Let $X$ be any non-basis of $M(K_{r+1})$. Then $X$ is not spanning, and hence there is a 
	hyperplane  $H\in \ZZ$ such that $X\subseteq H$. Then $H$ covers $X$, as $|X|=r$ and $r(H)=r-1$.
	The size of $\ZZ$ equals the number of hyperplanes of $M(K_{r+1})$, which is $2^{r}-1$.
		
	In order to show that $\ZZ$ is of minimum size, we apply Lemma \ref{lemma:dual}. 
	It suffices to find a collection $\mathcal{X}$ containing $2^r-1$ non-bases, such that no flat covers more than one of the non-bases in $\mathcal{X}$.
	For each  bipartition $\{U,W\}$ of $V$ such that neither $U$ nor $W$ is empty, we construct a non-basis $X_{\{U,W\}}$ as follows. 
	Without loss of generality, we may assume that $|U|\geq |W|$. Then $|U|\geq 3$, and we may pick a circuit  $C_U$ spanning $U$ and a tree $T_W$ spanning $W$. 
	Then we put $X_{\{U,W\}}:=C_U\cup T_W$. This yields a collection of $2^r-1$ non-bases $\mathcal{X}$, and we will argue that no flat $F$ covers more than one of these $X_{\{U,W\}}$. 
	
	Note that if a flat $F$ covers $X_{\{U,W\}}$, then $F$ must contain $C_U$ and be disjoint from $\delta(V(C))$. 
	So suppose for a contradiction that some flat $F$ covers both $X_{\{U_1,W_1\}}$ and $X_{\{U_2,W_2\}}$. As each $U_i$ contains at least half of the vertices in $V$, we have  $|U_1|+|U_2|\geq |V|$. 
	Since $U_1\neq W_2$, we have $U_1\cap U_2\neq \emptyset$. Since $F$ is disjoint from $\delta(V(C_1)$, and $C_2\subseteq F$, we have $U_2\subseteq U_1$, and vice versa we have $U_1\subseteq U_2$.
	So $U_1=U_2$, a contradiction.
	\end{proof}
The cover complexity of graphic matroids seems to be quite high, considering that any such matroid $M$ admits a very compact description: a graph, which takes no more than $\binom{r_M}{2}$ bits to describe. For cycle matroids of complete graphs, a minimal flat cover is far from being the most compressed description possible.

\section{\label{sec:small}Asymptotically small minor-closed classes}

The goal of this section is to put the machinery developed in Section \ref{sec:cover} to use in proving Theorem \ref{thm:asymptotic_small}. In fact, we will prove the following theorem, to which Theorem \ref{thm:asymptotic_small} is a corollary.

\begin{theorem}\label{thm:small_excluded} If $N=U_{2,k}$ for some $k \ge 2$, or if $N$ is one of $ U_{3,6}$, $P_6$, $Q_6$ or $R_6$, then
 $$\log |\Ex{N} \cap \MM_n| \le O\left(\frac{2^n}{n^{5/2}} \log^2(n)\right).$$\end{theorem}

\begin{proof}[Proof of Theorem \ref{thm:asymptotic_small} from Theorem \ref{thm:small_excluded}]
	We have
	$$\log(\lim_{n \to \infty} \frac{|Ex(N) \cap \MM_n|}{m_n})=\lim_{n \to \infty} (\log |Ex(N) \cap \MM_n|- \log m_n)= -\infty$$
	by Theorem \ref{thm:small_excluded} and the lower bound on $m_n$ due to Knuth (Theorem \ref{theorem:knuth}).
\end{proof}

To prove Theorem \ref{thm:small_excluded}, we will apply Theorem \ref{theorem:count} with $\epsilon=1$ and $s=r(N)$ for each $N$. 
So it will suffice to bound the cover complexity of matroids of rank $r(N)$ in each class $\Ex{N}$, and give a constant upper bound for $N=U_{2,4}$ and a linear bound for the cases $N=U_{3,6}, P_6, Q_6, R_6$. We prove these bounds in the remainder of this section.

First, we consider $U_{2,k}$.
\begin{lemma}\label{lemma:u2k_rank2}
	Let $M \in \MM_{n,2}$. If $M \not\succcurlyeq U_{2,k}$, then $\kappa(M) \le k$.
\end{lemma}
\begin{proof}
	As $M$ has rank 2, we can write $E(M) = E_0 \dot\cup E_1 \dot\cup \ldots \dot\cup E_m$, such that
	\begin{enumerate}
		\item $E_i \neq \emptyset$ for all $i > 0$, and
		\item $\{a,b\}$ is a basis of $M$ if and only if $a \in E_i$ and $b \in E_j$ for distinct $i, j > 0$.
	\end{enumerate}
	Then $r(E_0)=0$ and $r(E_0\cup E_1)=1$ for each $i$. Hence, $\ZZ:=\{E_0\}\cup\{E_0\cup E_i\mid i=1,\ldots, m\}$ is a flat cover of $M$, and so $\kappa(M) \le m+1$. Picking an element $e_i \in E_i$ for each $i = 1, 2, \ldots, m$, we find that the restriction of $M$ to $\{e_1, e_2, \ldots, e_m\}$ is a $U_{2,m}$-minor of $M$, so $m \le k-1$. We conclude that $\kappa(M) \le m+1 \le k$.
\end{proof}
For the case $N=U_{3,6}$, $P_6$, $Q_6$ or $R_6$, we bound the cover complexity of matroids of rank 3 in $\Ex{N}$. 
We will first argue that it is enough to bound the number of long lines in simple matroids of rank 3. Here, a {\em long line} is a flat of rank 1 containing at least 3 elements. The simplification of $M$ is denoted as $\si{M}$.

\begin{lemma}\label{lemma:long_lines} Let $M\in \MM_{n,3}$ be such that $\si{M}$ has  $L$ long lines. Then $\kappa(M)\leq 1+n/2+L$.\end{lemma}
\proof Consider the collection of flats $\ZZ \defeq \ZZ_0\cup\ZZ_1\cup\ZZ_2$ of $M$, where $\ZZ_0:=\{\cl_M(\emptyset)\}$, 
$\ZZ_1:=\{F\in\FF(M)\mid r_M(F)=1, |F|>1\}$, and 
$\ZZ_2:=\{F\in\FF(M)\mid r_M(F)=2, F\text{ contains at least $3$ flats of rank  }1\}.$
Then $\ZZ$ is a flat cover of $M$. For if $X$ is a non-basis of $M$, then either $X$ contains a loop, in which case $X$ is covered by $\ZZ_0$, or it contains a pair of parallel elements, so that $X$ is covered by $\ZZ_1$, and in the remaining case $X$ is covered by $\cl(X)\in\ZZ_2$.  Clearly $|\ZZ_0|=1, |\ZZ_1|\leq n/2, |\ZZ_2|=L$, and so $|\ZZ|\leq 1+n/2+L$.  \endproof 

\begin{lemma}\label{lemma:exclude_uniform_3}
	Let $M$ be a simple matroid of rank 3 with a $U_{3,k}$-minor, but without a $U_{3,k+1}$-minor. If $|E(M)|>\left(\binom{k-1}{2} - 1\right)\binom{k}{2} + k$, then $M$ has a line $\ell$ so that $M \setminus \ell$ has no $U_{3,k-1}$-minor.
	
\ignore{	
	Then one of the following holds:
	\begin{enumerate}[(i)]
		\begin{item}
			$M$ has a line $\ell$ so that $M \setminus \ell$ has no $U_{3,k-1}$-minor.
		\end{item}
		\begin{item}
			$M$ has at most $\left(\binom{k-1}{2} - 1\right)\binom{k}{2} + k$ elements.
		\end{item}
	\end{enumerate}
	}
\end{lemma}

\begin{proof}
	Suppose that $X \subseteq E(M)$ is such that $M$ restricted to $X$ is isomorphic to $U_{3,k}$. It follows that each element $e \in E(M)$ is on some flat spanned by a pair of elements in $X$. If each line $\ell$ spanned by 2 elements from $X$ contains at most $\binom{k-1}{2} + 1$ elements, then 
$M$ has at most $\left(\binom{k-1}{2} - 1\right)\binom{k}{2} + k$ elements.	
	
	So let $\ell$ be a line spanned by $x,y\in X$ with at least $\binom{k-1}{2} + 2$ elements. Then $M\setminus\ell$ cannot contain a $U_{3,k-1}$-minor. For if $M\setminus\ell$ does contain a $U_{3,k-1}$-minor, obtained by restricting to $X'$, then there exist 2 elements $u,v \in \ell$ that are not contained in any line spanned by two elements from $X'$. It follows that $M$ restricted to $X' \cup \{u, v\}$ is isomorphic to $U_{3,k+1}$. \end{proof}

\ignore{The following corollary lists some of the consequences of Lemma \ref{lemma:exclude_uniform_3}.
\begin{corollary}\label{crl:exclude_uniform_3}
	Let $M \in \MM_{n,3}$ be a simple matroid.
	\begin{enumerate}[(i)]
		\item If $M \not\succcurlyeq U_{3,4}$, then $M$ contains a line $\ell$ such that $r(M\setminus \ell) \le 1$.
		\item If $M \not\succcurlyeq U_{3,5}$ and $n > 16$, then $M$ contains a line $\ell$ such that $r(M \setminus \ell) \le 2$.
		\item If $M \not\succcurlyeq U_{3,6}$ and $n > 55$, then $M$ contains a line $\ell$ such that $M \setminus \ell \not\succcurlyeq U_{3,4}$.
		\item If $M \not\succcurlyeq U_{3,7}$ and $n > 141$, then $M$ contains a line $\ell$ such that $M \setminus \ell \not\succcurlyeq U_{3,5}$.
	\end{enumerate}
\end{corollary}
}
We now settle the case $N=U_{3,6}$.
\begin{lemma} \label{lemma:u36_rank3}
	Let $M \in \MM_{n,3}$. If $M \not\succcurlyeq U_{3,6}$, then $\kappa(M) \le 496+n$.
\end{lemma}
\proof Suppose $MM \in \MM_{n,3}$ and $M \not\succcurlyeq U_{3,6}$. If $M':=\si{M}$ has at most 55 elements, then $M'$ has at most $\binom{55}{2}/3=495$ long lines, so that then $\kappa(M)\leq 1+n/2+495$ by Lemma \ref{lemma:long_lines}. So we have $|E(M')|>55$.

If $M$ has  no $U_{3,6}$-minor, then $M'$ contains a line $\ell$ so that $M'\setminus \ell\not\succcurlyeq U_{3,4}$  by Lemma \ref{lemma:exclude_uniform_3} (with $k\leq 5$). 
A second application of Lemma \ref{lemma:exclude_uniform_3} (with $k\leq 3$) shows that $M'\setminus \ell$ contains a line $\ell'$ such that $M' \setminus \ell'\setminus \ell''$ has rank at most 1.  
So either $E(M')=\ell\cup\ell'$, in which case $\ell$ and $\ell'$ are the only long lines  of $M'$, or $E(M') = \ell\dot\cup\ell'\dot\cup\{e\}$  for some element $e\in E(M')$. 
In the latter case any long line except $\ell$ and $\ell'$ contains $e$, some $f \in \ell$ and some $g \in \ell'$. As each of these long lines is determined by $\{e,f\}$ and $\{e,g\}$, the number of such lines is at most $\min\{|\ell|, |\ell'|\} \le (|E(M')|-1)/2$. So then $M'$ has at most $2+n/2$ long lines. Then $\kappa(M)\leq 1+n/2 +(2+n/2)$ by Lemma \ref{lemma:long_lines}.
\endproof
Next, we consider $N = Q_6$.
\begin{lemma} \label{lemma:Q6_rank3}
	Let $M \in \MM_{n,3}$. If $M \not\succcurlyeq Q_{6}$, then $\kappa(M) \le 41 + n$.
\end{lemma}
\proof Let $M':=\si{M}$. If $M'$ does not have any intersecting pair of long lines, then $M'$ has no more than  $n/3$ long lines, and then $\kappa(M)\leq 1+n/2+n/3$ by Lemma \ref{lemma:long_lines}. So suppose $M'$ has two long lines $\ell$ and $\ell'$ that intersect in point $e$. If $E(M')=\ell\cup\ell'$, then $\ell$ and $\ell'$ are the only two long lines of $M'$, and then $\kappa(M)\leq 1+n/2+2$. So consider a point 
$f\in E(M')\setminus (\ell\cup\ell')$. Each point $p\in \ell\setminus\{e\}$ determines a line $\ell_p$ through $p$ and $f$ that intersects $\ell'$ in at most one point, so that if $|\ell'|>4$, we may obtain a $Q_6$-restriction of $M'$ on $\{e,f,p,q,p',q'\}$ by arbitrarily taking $p,q\in \ell\setminus\{e\}$  and $p',q'\in\ell'\setminus(\{e\}\cup\ell_p\cup\ell_q)$. So $|\ell'|\leq 4$, and by symmetry $|\ell|\leq 4$ as well. 

If $|\ell'|=4$, and there is some $f\in E(M')\setminus (\ell\cup\ell')$ that is on at most one line $\ell''$ which intersects both $\ell\setminus\{e\}$ and $\ell'\setminus\{e\}$, then we may choose $\{e,f,p,q,p',q'\}$ spanning a $Q_6$-minor as before. So then each point $f$ is determined by two such lines, and hence there are no more than $3^2$ points in $E(M')\setminus (\ell\cup\ell')$. Then $|E(M')|\leq 16$, and $M'$ has no more than $\binom{16}{2}/3=40$ long lines, so that $\kappa(M)\leq 1+n/2+40$.

So $|\ell'|=3$, and indeed every long line of $M'$ that intersects another has length 3. Moreover, each $f\in E(M')\setminus (\ell\cup\ell')$ is on some line $\ell''$ which intersects both $\ell\setminus\{e\}$ and $\ell'\setminus\{e\}$, and each such line has length 3. There are no more than 4 such lines $\ell''$, so that there are at most 4 points $f\in E(M')\setminus (\ell\cup\ell')$. Then the number on points in $M'$ is at most 9, and the number of long lines is at most $\binom{9}{2}/3=12$. Then $\kappa(M)\leq 1+n/2+12$.\endproof

Finally, we prove the bounds for $P_6$ and $R_6$.
\begin{lemma} \label{lemma:R6_rank3}
	Let $M \in \MM_{n,3}$. If $M \not\succcurlyeq R_6$, then $\kappa(M) \le 13+n/2$.
\end{lemma}
\proof Let $M':=\si{M}$. As $M'$ has no $R_6$-minor,  any two long lines of $M'$ must have a common point.  Let $\ell$ be a longest line of $M'$.  
Then any long line $\ell'$ of $M'$ other than $\ell$ has length 3, for if $\ell'$ has 4 or more points, then so has $\ell$, and then three elements on $\ell\setminus \ell'$ with three elements from $\ell'\setminus \ell$ will then form an $R_6$-restriction. If $M'$ has two long lines $\ell'$ and $\ell''$ distinct from $\ell$, then any remaining line other than $\ell$ must intersect both $\ell'$ and $\ell''$, and both lines have 3 points, so then $M'$ has at most $3+3^2$ long lines. But if $M'$ does not have  two long lines $\ell'$ and $\ell''$ distinct from $\ell$, then $M'$ has at most 2 long lines. By Lemma \ref{lemma:long_lines}, we have $\kappa(M)\leq 1+n/2+12$.\endproof
\begin{lemma}\label{lemma:P6_rank3}
	Let $M \in \MM_{n,3}$. If $M \not\succcurlyeq P_6$, then $\kappa(M) \le 13+19n$.
\end{lemma}
\begin{proof}
	If $M':=\si{M}$ does not have an $R_6$-minor, then neither does $M$, and then $\kappa(M) \le 13+n/2$ by the previous lemma. If $M'$ does have an $R_6$-minor, fix one such minor with the two lines $\{e_1, e_2, e_3\}$ and $\{f_1, f_2, f_3\}$. There are 9 lines spanned by pairs $\{e_i, f_j\}$, and hence the set $U$ of  intersection points between pairs of these lines contains at most $\binom{9}{2}=36$ points of $M'$. 
	Let $\ell:=\cl_{M'}\{e_1, e_2, e_3\},\ell':=\cl_{M'}\{f_1, f_2, f_3\}$ be the two long lines of $M'$ spanned by the lines of the $R_6$-minor. If $M'$ contains any element $g$ not in $\ell\cup\ell'\cup U$, then $\{e_1, e_2, e_3, f_1, f_2, f_3, g\}$ contains a $P_6$-minor. So each long line of $M'$ other than $\ell, \ell'$  intersects $U$. It follows that $M'$ has at most $2+(36(n/2))$ long lines, and hence that in this case $\kappa(M)\leq 1+n/2 +(2+18n)\leq 3+19n$ by Lemma \ref{lemma:long_lines}.
\end{proof}

\section{\label{sec:large}A large minor-closed class}
\subsection{Stable sets in the Johnson graph}

The {\em Johnson graph} $J(n,r)$ is the graph with vertex set $\binom{[n]}{r}$, in which two vertices are connected by an edge if and only if they intersect in $r-1$ elements.  

Recall that a matroid $M \in \MM_{n,r}$ is called sparse paving if any $r$-set is either a basis or a circuit-hyperplane in $M$. Write $\SS_{n,r}$ for the collection of sparse-paving matroids in $\MM_{n,r}$. The following lemma states that there is a 1-1 correspondence between stable sets in $J(n,r)$ and $\SS_{n,r}$. It was essentially shown by Piff and Welsh \cite{PiffWelsh1971}, in proving an upper bound on the number of sparse paving matroids.

\begin{lemma}\label{lemma:stablesets} Let $0<r<n$. If $U \subseteq \binom{[n]}{r}$, then $U$ is a stable set in $J(n,r)$ if and only if $\binom{[n]}{r} \setminus U$ is the set of bases of a sparse-paving matroid.
\end{lemma}
In order to avoid trivialities, we will assume $0 < r < n$ in what follows. The following construction due to Graham and Sloane \cite{GrahamSloane} shows that the Johnson graph $J(n,r)$ has an $n$-coloring. Consider the function $f:V(J(n,r)) \to \Z/n\Z$, given by
\begin{equation*}
	f:X \mapsto \sum_{x \in X} x \mod n.
\end{equation*}
 If $X$ and $Y$ are adjacent vertices of $J(n,r)$, then $Y=X+e-f$ for some distinct $e,f\in [n]$, and  then $f(Y)=f(X+e-f)=f(X)+e-f\neq f(X)$. Therefore, the set of vertices 
\begin{equation*}
	U_{n,r,k} \defeq \left\{ X \in V(J(n,r)) : f(X) = k \mod n\right\}
\end{equation*}
  is a stable set of $J(n,r)$ for each $k\in \Z/n\Z$. Define
\begin{equation}
	\SS_{n,r,k} \defeq \left\{ \left([n], \binom{[n]}{r} \setminus U\right) : U \subseteq U_{n,r,k}\right\}.
\end{equation}
By Lemma \ref{lemma:stablesets}, each $\SS_{n,r,k}$ is a collection of matroids. Note that if $M \in \SS_{n,r,k}$, then any circuit-hyperplane of $M$ is in $U_{n,r,k}$.
\begin{lemma}\label{lemma:big_S}If $0<r<n$, then there is some $k\in \Z/n\Z$ such that  $|\SS_{n,r,k}|\geq 2^{\frac{1}{n}\binom{n}{r}}$.
\end{lemma}
\proof As the $n$ sets $U_{n,r,k}$ partition the $\binom{n}{r}$ vertices of $J(n,r)$, there is a $k$ such that $|U_{n,r,k}|\geq \frac{1}{n}\binom{n}{r}$. 
For that $k$, we have $|\SS_{n,r,k}|\geq 2^{\frac{1}{n}\binom{n}{r}}$. \endproof
Theorem \ref{theorem:knuth}, Knuths lower bound on the number of matroids, follows by taking $r=\lfloor n/2\rfloor$ in the above lemma. It was noted earlier in \cite{MayhewWelsh2013} that the construction of Graham and Sloane gives an improvement over Knuths original argument.

In what follows, we argue that many of the matroids in $\SS_{n,r,k}$ do not have an $M(K_4)$-minor, and that there exist matroids in that class without a $V_8$-minor and with large cover complexity.

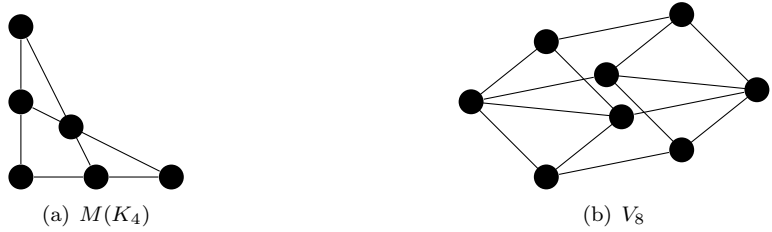
\begin{figure}[tb]
	\subfigure[$M(K_4)$]{
		\begin{tikzpicture}[scale=0.8]
			\node[circle, fill=black] (A) {};
			\node[circle, fill=black, below of=A] (B) {};
			\node[circle, fill=black, below of=B] (C) {};
			\node[circle, fill=black, right of=C] (D) {};
			\node[circle, fill=black, right of=D] (E) {};
			
			\draw (A) -- (B) -- (C);
			\draw (C) -- (D) -- (E);
			\draw[name path=BE] (B) -- (E);
			\draw[name path=AD] (A) -- (D);
			
			\path[name intersections={of=BE and AD,by=F}];
			
			\node[circle, fill=black] at (F) {};
		\end{tikzpicture}
	}\hspace{.2\textwidth}
	\subfigure[$V_8$]{
		\begin{tikzpicture}[scale=1,x={(1,-1)},y={(1,1.8)},z={(0.2,0.4)}]
			\begin{scope}[canvas is xy plane at z=0]
				\draw (0,0) node[circle,fill=black] (A1) {} -- (1,0) node[circle,fill=black] (B1) {} -- (1,1) node[circle,fill=black] (C1) {} -- (0,1) node[circle,fill=black] (D1) {} -- (0,0) -- (1,1);
			\end{scope}
			\begin{scope}[canvas is xy plane at z=-3]
				\draw (0,0) node[circle,fill=black] (A2) {} -- (1,0) node[circle,fill=black] (B2) {} -- (1,1) node[circle,fill=black] (C2) {} -- (0,1) node[circle,fill=black] (D2) {} -- (0,0) -- (1,1);
			\end{scope}
			\foreach \i in {A,B,C,D}{\draw (\i1) -- (\i2);}
		\end{tikzpicture}
	}
	\caption{\label{fig:MK4V8}Two sparse paving matroids.}
\end{figure}

\subsection{Matroids without an $M(K_4)$-minor}
Recall that $M(K_4)$ is the graphical matroid obtained from the complete graph on 4 vertices. It is a matroid of rank 3 on 6 elements. See Figure \ref{fig:MK4V8}(a).
\begin{lemma}\label{lemma:MK4}
	Let $n \in \NN$ be an odd integer and let $r, k \le n$. Then no $M \in \SS_{n,r,k}$ has $M(K_4)$ as a minor.
\end{lemma}
\begin{proof}
	Let $M \in \SS_{n,r,k}$, and suppose to the contrary that $M/A\setminus B\cong M(K_4)$. Without loss of generality $A$ is an independent set, so $M$ has circuit-hyperplanes
	\begin{equation*}
		A \cup \{a,b,c\}, \quad A\cup\{a,d,e\}, \quad A\cup\{c,d,f\}, \quad\text{and}\quad A\cup\{b,e,f\},
	\end{equation*}
	for some distinct $a, b, c, d, e, f \in [n]$ such that $\{a,b,c,d,e,f\} = [n] \setminus (A \cup B)$. As each circuit-hyperplane of $M$ is in $U_{n,r,k}$, it follows that
	\begin{equation*}
		a + b + c = a + d + e = c + d + f = b + e + f = k - \sum_{x \in A} x \mod n.
	\end{equation*}
	Adding the first and third expression and subtracting the second and fourth equation yields $2(c-e) = 0 \mod n$. As $n$ is odd, 2 has a multiplicative inverse in $\Z/n\Z$ and it follows that $c = e \mod n$. As $c, e \in [n]$, this contradicts $c \neq e$.
\end{proof}
Note that sparse paving matroids with $M(K_4)$ as a minor do exist. In particular, $M(K_4)$ itself is sparse paving.
\begin{corollary}$\log |\Ex{M(K_4)}\cap\MM_n|\geq \frac{1}{n}\binom{n}{\lfloor n/2\rfloor}(1-o(1))$ as $n\rightarrow \infty$.
\end{corollary}
\proof By the Lemma, we have  $\log |\Ex{M(K_4)}\cap\MM_n|\geq \log |\SS_{n,r,k}|$ for each $k$ if $n$ is odd, and $$\log |\Ex{M(K_4)}\cap\MM_n|\geq \log |\SS_{n-1,r,k}|$$ for each $k$ if $n$ is even, as each matroid in $\SS_{n-1,r,k}$ can be extended by a loop to a matroid in $\MM_n$. By Lemma \ref{lemma:big_S}, we have $\log |\SS_{n,\lfloor n/2\rfloor,k}|\geq \frac{1}{n}\binom{n}{\lfloor n/2\rfloor}$ for some $k$. The corollary follows.\endproof
This proves Theorem \ref{theorem:k4}, which is a restatement of the corollary.

\subsection{Matroids without a $V_8$-minor} Recall that the V\'{a}mos matroid $V_8$ is a matroid of rank 4 on 8 elements $E(V_8) = \{a,a',b,b',c,c',d,d'\}$, of which the set of bases consists of each 4-set, except the following five:
\begin{equation}\label{eq:V8_circuithyperplanes}
	\{a,a',b,b'\}, \quad \{a,a',c,c'\}, \quad \{a, a', d, d'\}, \quad \{b, b', c, c'\}, \quad\text{and} \quad \{b, b', d, d'\}.
\end{equation}
$V_8$ is depicted in Figure \ref{fig:MK4V8}(b).
\begin{lemma}\label{lemma:v8kappa}
	Let $n \in \NN$ be an odd integer and let $r \le n$. Then there exists a sparse paving matroid $M \in \MM_{n,r}$ such that $M\not \succeq V_8$ and $\kappa(M)\geq \binom{n}{r}/n$.
\end{lemma}
\begin{proof}
	If $M \in \SS_{n,r,k}$ is the matroid whose set of circuit-hyperplanes is $U_{n,r,k}$, then $\kappa(M)=|U_{n,r,k}|\geq \binom{n}{r}/n$ for an appropriate choice of $k$. We will argue that $M$ does not have $V_8$ as a minor. Suppose to the contrary that $M/A\setminus B \cong V_8$ (with $A$ independent in $M$). Then for distinct elements $a, a', b, b', c, c', d,d' \in [n]$, $M$ has five circuit-hyperplanes of the form $A \cup X$, with $X$ any of the five sets in Equation \eqref{eq:V8_circuithyperplanes}. It follows that
	\begin{equation*}
		a + a' + b + b' = a + a' + c + c' = b + b' + c + c' = k - \sum_{x \in A} x \mod n.
	\end{equation*}
	Writing $k' = k - \sum_{x \in A} x$, we find that $2(a + a') = 2(b + b') = 2(c + c') = k' \mod n$. Similarly, it follows from
	\begin{equation*}
		a + a' + b + b' = a + a' + d + d' = b + b' + d + d' = k' \mod n
	\end{equation*}
	that $2(d + d') = k' \mod n$. As 2 has a multiplicative inverse in $\Z/n\Z$, it follows that $c + c' + d + d' = k' \mod n$, and hence $\{c, c', d, d'\}$ is a circuit-hyperplane of $M/A\setminus B \cong V_8$: a contradiction.
\end{proof}

\section{\label{sec:outtro} Final remarks}
\subsection{The Blow-up Lemma and entropy}The Blow-up Lemma bears a striking similarity to the following result, which was derived in  \cite{BansalPendavinghVanderPol2012B} from Shearers' entropy Lemma.
\begin{lemma}\label{contraction}For $0\leq t\leq r\leq n$, we have 
$\frac{1}{\binom{n}{r}} \log(m_{n,r}+1) \leq \frac{1}{\binom{n-t}{r-t}} \log(m_{n-t, r-t}+1).$
\end{lemma}
As is noted there, the argument actually applies to count matroids in any contraction-closed class $\Mcal$, so that we  have
\begin{equation}\label{entropy}\frac{1}{\binom{n}{r}} \log(m'_{n,r}+1) \leq \frac{1}{\binom{n-t}{r-t}} \log(m'_{n-t, r-t}+1)\end{equation}
for the number $m'_{n,r}:=|\Mcal\cap\MM_{n,r}|$ . 

However, it does not seem to be feasible to prove that almost all matroids have a $U_{2,k}$-minor using \eqref{entropy} with $\Mcal=\Ex{U_{2,k}}$. 
The number of matroids without an $U_{2,k}$-minor on a fixed ground set with $n$ elements is at least $(k-1)^{n-k+1}/k!$, and so $\log(m'_{n,2})=\Theta(n)$. 
So at best, the upper bound on $m'_{n,r}$ we may  obtain from \eqref{entropy} is of the same order as the lower bound of Knuth.

The difference between the Blow-up Lemma and \eqref{entropy} is subtle. For a minor-closed class of matroids $\Mcal$ in which the number of distinct matroids of low rank $r$ is low and the maximal cover complexity among such matroids is high, applying \eqref{entropy} may succeed where the Blow-up Lemma might fail.

\subsection{The problem with $V_8$ and $U_{3,7}$}
Our present method for proving cases of Conjecture \ref{conj:paving} applies to any matroid $N$ so that for some $s$ and some $\epsilon>0$, we can show that
\begin{equation}\label{goodbound}\max\{\kappa(M)\mid M\in \Ex{N}\cap\MM_{n,s}\}\leq O(n^{s-1-\epsilon}).\end{equation}
For then, we may apply  Theorem \ref{theorem:count} with $\epsilon>0$, and obtain
$$\log |\Ex{N} \cap \MM_n| \le O\left(\frac{2^n}{n^{3/2+\epsilon}} \log^2(n)\right).$$
With $\epsilon>0$, this is just enough to beat Knuth's lower bound and to prove that $\Ex{N}$ is a vanishing class.

But for many $N$, we will not be able to prove \eqref{goodbound} with $\epsilon>0$. For the V{\'a}mos matroid, for example, we know that 
$$\max\{\kappa(M)\mid M\in \Ex{V_8}\cap\MM_{n,s}\}\geq o(n^{s-1})$$
for any $s$, as for each odd $n$ there will be a sparse paving matroid $M\in \Ex{V_8}\cap\MM_{n,s}$ with $\kappa(M)\geq\binom{n}{s}/n$ by Lemma \ref{lemma:v8kappa}.

The arguments that apply to the V\'amos matroid will apply to many sparse paving matroids with sufficiently many intersecting circuit-hyperplanes. But uniform matroids are safe from such attacks, and we think that they presently are the most approachable cases of Conjecture \ref{conj:paving}. The smallest uniform matroid that we have not settled is $N=U_{3,7}$. Unfortunately, we also have
$$\max\{\kappa(M)\mid M\in \Ex{U_{3,7}}\cap\MM_{n,s}\}\geq o(n^{s-1})$$
for all $s$. For note that the Dowling matroid of rank $s$ over $GF(q)$ does not have a $U_{3,7}$-minor, has $\Theta(q)$ elements for fixed $s$, and has $\kappa(M)\geq o(q^{s-1})$. For the connection between Dowling matroids and excluding uniform minors, see e.g. \cite{Geelen2008}.

In summary, Theorem \ref{theorem:count} is not sharp enough to handle the case $\epsilon=0$, which evidently does occur. But the proof of this theorem does not yet incorporate the techniques that are applied for the sharpest upper bound on the number of matroids in \cite{BansalPendavinghVanderPol2012}. The following better bound may be possible.
\begin{conjecture} \label{theorem:count}Let $\Mcal$ be a contraction-closed class of matroids. If for some natural number $s$, and positive constants $c$ and $\epsilon$ we have 
$\max\{\kappa(M)\mid M\in \Mcal \cap \MM_{n,s}\}\leq  cn^{s-1-\epsilon}$, then
$$\log |\Mcal \cap \MM_{n}|\leq  O(\frac{1}{n^{3/2+\epsilon}}2^n)  \text{ as } n\rightarrow\infty.$$
\end{conjecture}

\ignore{
\begin{figure}[tb]
	\begin{tikzpicture}
		\path (0,0) coordinate (A) -- (-60:4) coordinate (B) -- (-120:4) coordinate (C) -- cycle;
		
		\foreach\i/\p/\t in {A/right/1 \\ 0 \\ 0,B/right/0 \\ 1 \\ 0,C/left/0 \\ 0 \\ 1}{
			\node[circle, fill=black] () at (\i) {};
			\node[\p] at (\i) {$\begin{bmatrix} \t \end{bmatrix}$};
		}
		\foreach\i/\j in {A/B, C/B, C/A} { 
			\node[circle, fill=black] () at ($(\i)!0.25!(\j)$) {};
			\node[circle, fill=black] () at ($(\i)!0.75!(\j)$) {};
			\draw (\i) -- (\j);
			\draw($($(\i)!0.5!(\j)$)!0.15!90:(\j)$) node () {$\ldots$};
		}
		\draw[decorate,decoration={brace}, ultra thick] ($(B)+(0,-0.3)$) -- ($(C)+(0,-0.3)$) node[midway,below](bracket text){$q+1$ points};
	\end{tikzpicture}
	\caption{\label{fig:Mq}The matroid $M_q$.}
\end{figure}
}

 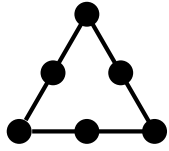
\begin{figure}[tb]
	
		\begin{tikzpicture}[scale=0.9]
			\node[dot] (A) {};
			\draw[line] (A) -- (0:1) coordinate (B) -- (0:2) coordinate (C);
			\draw[line] (A) -- (60:1) coordinate (D) -- (60:2) coordinate (E);
			\draw[line] (C) -- ($(C)!0.5!(E)$) coordinate (F) -- (E);
			\foreach\id in {B,C,D,E,F} { \node[dot, at=(\id)] {}; }
		\end{tikzpicture}
	\caption{\label{fig:Q6W3}$W^3$, the remaining sparse paving matroid on six elements.}
\end{figure}
 \subsection{A challenge} Theorem \ref{thm:asymptotic_small} covers all sparse paving matroids of rank 3 on 6 elements except $W^3$, the whirl (see Figure \ref{fig:Q6W3}).
We were unable to prove sufficient bounds on $\max\{\kappa(M)\mid M\in\Ex{N}\cap \MM_{n,3}\}$ as a function of $n$. 
We leave this as a challenge to the reader. 
 
\section{Acknowledgement}
We have recently worked with Nikhil Bansal on \cite{BansalPendavinghVanderPol2012B} and \cite{BansalPendavinghVanderPol2012}, and we have had the benefit of many stimulating discussions with him on that closely related topic. We like to point out that the argument in the proof of Lemma \ref{lemma:kappa_rounding} was part of an early version of our joint paper \cite{BansalPendavinghVanderPol2012}.
\bibliographystyle{plain}
\bibliography{complexity}

\end{document}